\documentclass{amsart}
\usepackage{amsmath,amssymb}

\newtheorem{theorem}{Theorem}[section]
\newtheorem{definition}[theorem]{Definition}

\newtheorem{proposition}[theorem]{Proposition}

\numberwithin{equation}{section}

\newcommand{\R}{\mathbb{R}}
\newcommand{\N}{\mathbb{N}}
\newcommand{\Z}{\mathbb{Z}}
\newcommand{\C}{\mathbb{C}}

\newcommand{\dis}{\displaystyle}
\newcommand{\norm}[1]{\left\Vert#1\right\Vert}
\newcommand{\eps}{\varepsilon}

\begin{document}

\title[The multivariate integer Chebyshev problem]
{The multivariate integer Chebyshev problem}%
\author{P. B. Borwein and I. E. Pritsker}%

\thanks{Research of P. B. Borwein was partially supported by NSERC of Canada and MITACS}
\thanks{Research of I. E. Pritsker was partially supported by the National Security Agency under grant H98230-06-1-0055, and by the Alexander von Humboldt Foundation.}

\address{Department of Mathematics and Statistics, Simon Fraser
University, Burnaby, B. C., V5A 1S6, Canada}%
\email{pborwein@cecm.sfu.ca}

\address{Department of Mathematics, Oklahoma State University,
Stillwater, OK 74078, U.S.A.}%
\email{igor@math.okstate.edu}




\begin{abstract}

The multivariate integer Chebyshev problem is to find polynomials with integer coefficients that minimize the supremum norm over a compact set in $\C^d.$ We study this problem on general sets, but devote special attention to product sets such as cube and polydisk. We also establish a multivariate analog of the Hilbert-Fekete upper bound for the integer Chebyshev constant, which depends on the dimension of space. In the case of single variable polynomials in the complex plane, our estimate coincides with the Hilbert-Fekete result.

\end{abstract}

\maketitle


\section{The integer Chebyshev problem and its multivariate counterpart}

The supremum norm on a compact set  $E \subset {\C}^d,\ d\in\N,$ is
defined by
\[
\norm{f}_{E} := \sup_{z \in E} |f(z)|.
\]
We study polynomials with integer coefficients that minimize the sup norm on a set $E$, and investigate their asymptotic behavior. The univariate case $(d=1)$ has a long history, but the problem is virtually untouched for $d\ge 2.$ Let ${\mathcal P}_n ( {\C})$ and ${\mathcal P}_n ( {\Z})$ be the classes of algebraic polynomials in one variable, of degree at most $n$, respectively with complex and with integer coefficients. The problem of minimizing the uniform norm on $E$ by {\it monic} polynomials from ${\mathcal P}_n ( {\C})$ is the classical {\em Chebyshev problem} (see \cite{BE95}, \cite{Ri90}, \cite{Ts75}, etc.) For $E=[-1,1]$, the explicit solution of this problem is given by the monic Chebyshev polynomial of degree $n$:
\[
T_n (x) := 2^{1 -n} \cos (n \arccos x), \quad n \in \N.
\]
By a linear change of variable, we immediately obtain that
\[
t_n (x) := \left( \frac{b -a}{2} \right)^n T_n \left( \frac{2x -a
-b}{b -a} \right)
\]
is a monic polynomial with real coefficients and the smallest norm
on $[a,b]\subset {\R}$ among all {\it monic} polynomials of degree
$n$ from ${\mathcal P}_n ({\C})$. In fact,
\begin{equation} \label{1.1}
\| t_n \|_{[a,b]} = 2 \left( \frac{b -a}{4} \right)^n, \quad n \in
{\N},
\end{equation}
and the {\it Chebyshev constant} for $[a,b]$ is given by
\begin{equation} \label{1.2}
t_{\C}([a,b]) := \lim_{n \rightarrow \infty} \| t_n \|_{[a,b]}^{1/n}
= \frac{b -a}{4}.
\end{equation}
The Chebyshev constant of an arbitrary compact set $E \subset {\C}$
is defined similarly:
\begin{equation} \label{1.3}
t_{\C}(E) := \lim_{n \rightarrow \infty} \| t_n \|_{E}^{1/n},
\end{equation}
where $t_n$ is the Chebyshev polynomial of degree $n$ on $E$. It is
known that $t_{\C}(E)$ is equal to the transfinite diameter and the
logarithmic capacity ${\rm cap}(E)$ of the set $E$ (cf. \cite{Ra95} and \cite[pp. 71-75]{Ts75} for definitions and background material).

We say that $Q_n \in {\mathcal P}_n ( {\Z})$ is an {\it integer
Chebyshev polynomial} for a compact set $E \subset \C$ if
\begin{equation} \label{1.5}
\| Q_n \|_{E} = \inf_{0 \not\equiv P_n \in{\mathcal P}_n ( {\Z})} \|
P_n \|_{E},
\end{equation}
where the $\inf$ is taken over all polynomials from ${\mathcal P}_n
( {\Z})$ that are not identically zero. The {\it integer Chebyshev
constant} (or integer transfinite diameter) for $E$ is given by
\begin{equation} \label{1.6}
t_{\Z}(E) := \lim_{n \rightarrow \infty} \| Q_n \|_{E}^{1/n}.
\end{equation}
In general, $0\le t_{\Z}(E)\le 1$, because $P_n(z)\equiv 1$ is one
of the competing polynomials for the $\inf$ in \eqref{1.5}. One may
readily observe that if $E=[a,b]$ and $b -a \geq 4$, then $Q_n (x)
\equiv 1,\ n \in {\N},$ by (\ref{1.1}) and (\ref{1.6}), so that
\begin{equation} \label{1.7}
t_{\Z}([a,b]) = 1, \quad b-a \ge 4.
\end{equation}
We also obtain directly from the definition and (\ref{1.2}) that
\begin{equation} \label{1.8}
\frac{b -a}{4} = t_{\C}([a,b]) \leq t_{\Z}([a,b]),\quad b-a<4.
\end{equation}
The results of Hilbert \cite{Hi1894} imply the important upper bound
\begin{equation} \label{1.9}
t_{\Z}([a,b]) \leq \sqrt{\frac{b-a}{4}}.
\end{equation}
These results were generalized to the case of an arbitrary compact
set $E \subset \C$ by Fekete \cite{Fe23}, who developed a new
analytic setting for the problem, by introducing the transfinite
diameter of $E$ and showing that it is equal to $t_{\C}(E).$ Both,
the transfinite diameter and the Chebyshev constant, were later
proved to be equal to the logarithmic capacity ${\rm cap}(E)$, by
Szeg\H{o} \cite{Sz24}. Therefore we state the result of Fekete as
follows:
\begin{equation} \label{1.10}
t_{\Z}(E) \leq \sqrt{t_{\C}(E)} = \sqrt{{\rm cap}(E)},
\end{equation}
where $E$ is $\R$-symmetric. It contains Hilbert's estimate
(\ref{1.9}) as a special case, since $t_{\C}([a,b])=(b-a)/4$ by
(\ref{1.2}). The following useful observation on the asymptotic
sharpness for the estimates (\ref{1.9}) and (\ref{1.10}) is due to
Trigub \cite{Tr71}. For the sequence of the intervals
$I_m:=[1/(m+4),1/m]$, we have $ t_{\Z}(I_m) \ge \frac{1}{m+2}$ and
\[
\lim_{m \to \infty} \frac{t_{\Z}(I_m)}{\sqrt{|I_m|/4}} = 1.
\]
Furthermore, it was shown in \cite{Pr1} that, for the circle
$L_{1/n}=\{z:|nz-1|=1/n\}, \ n\in\N,\ n\ge2$, we have
$t_{\Z}(L_{1/n})=1/n$ and $t_{\C}(L_{1/n})=1/n^2$. Hence {\it
equality} holds in \eqref{1.10} in this case.

The majority of lower bounds for the integer Chebyshev constant are obtained for intervals by using the resultant method, see \cite{Mo94,BE96,FRS97}. They heavily depend on the arithmetic properties of endpoints of the interval. Different methods based on weighted potential theory are employed in \cite{Pr1}. We note that the exact values of $t_{\Z}$ are not known for any segment of length less than $4$. On the other hand, close upper and lower bounds are available for many intervals, with $[0,1]$ being the most thoroughly studied.

Even the classical Chebyshev problem for multivariate polynomials is
considerably more complicated than its univariate version.
Concerning the multivariate integer Chebyshev problem, very little
is known at all. But some special cases of small integer polynomials
in many variables were certainly studied before. For example, this
problem received attention in light of the Gelfond-Schnirelman
method in the distribution of prime numbers (see Gelfond's comments
in \cite[pp. 285--288]{Cheb44}, and see \cite{Na82,Ch83,Pr2} for
further developments).

By ${\mathcal P}_n^d({\C})$ and ${\mathcal P}_n^d({\Z})$, we denote
the classes of algebraic polynomials in $d$ variables, of degree at
most $n\in\N_0:=\N\cup\{0\}$, respectively with complex and with
integer coefficients. The general form of such polynomials is as
follows:
\[
P_n(z)=\sum_{|k|\le n} a_k z^k, \quad z\in\C^d,
\]
where $k=(k_1,\ldots,k_d)\in \N_0^d,\ z^k=z_1^{k_1}\ldots
z_d^{k_d},$ and $|k|=\sum_{i=1}^d k_i.$

\begin{definition} \label{def1.1}
A multivariate integer Chebyshev polynomial $C_n\in{\mathcal
P}_n^d({\Z})$ for a compact set $E\subset\C^d$ is defined by
\[
\|C_n\|_E=\inf_{0 \not\equiv P_n \in{\mathcal P}_n^d({\Z})} \| P_n
\|_{E}
\]
We write $t_{\Z}(n,E):=\|C_n\|_E$, and define the multivariate
integer Chebyshev constant for $E$ as
\[
t_{\Z}(E) := \lim_{n \rightarrow \infty} \| C_n \|_{E}^{1/n}.
\]
\end{definition}

The multivariate integer Chebyshev constant is a monotone and
continuous set function, which is consistent with the classical
one-dimensional version.

\begin{proposition} \label{prop1.2}
Let $E\subset \C^d$ and $F \subset \C^d$ be compact sets.\\
{\rm (i)} If $E\subset F$ then $t_{\Z}(n,E) \le t_{\Z}(n,F),\
n\in\N_0,$
and $t_{\Z}(E) \le t_{\Z}(F).$\\
{\rm (ii)} Define $ E_{\delta}:=\bigcup_{w \in E} \{ z: |z-w|\le \delta\},$ where $|z-w|$ is the Euclidean distance in $\C^d$ (as $\R^{2d}$).
For any $\eps>0$ we can find $\delta>0$ such that
\[
0 \le t_{\Z}(E_{\delta}) - t_{\Z}(E) \le \eps.
\]
\end{proposition}

Another property similar to the univariate case states that if the
set is sufficiently large, then the integer Chebyshev polynomials
are given by $C_n(z)\equiv 1,\ n\in\N_0.$ We also estimate the multivariate integer Chebyshev constant of $E$ by the integer Chebyshev constants of its coordinate projections.

\begin{proposition} \label{prop1.3}
Suppose that $E_j\subset \C,\ j=1,\ldots,d,$ are compact sets, and
define $E:=E_1\times\ldots\times E_d.$ We have
\[
t_{\Z}(E) \le \min_{1\le j\le d} t_{\Z}(E_j).
\]
If $t_{\C}(E_j)\ge 1,\ j=1,\ldots,d,$ then $C_n(z)\equiv 1,\
n\in\N_0,$ and $t_{\Z}(E)=1.$
\end{proposition}
Note that if $E\subset\C^d$ is arbitrary, then we have $E\subset E_1\times\ldots\times E_d,$ where $E_j$ is a projection of $E$ onto the $j$th coordinate plane. Hence the estimate $t_{\Z}(E) \le \min_{1\le j\le d} t_{\Z}(E_j)$ is valid in this case too.

We now state a result on vanishing of the multivariate integer
Chebyshev polynomials on the product lattice of algebraic
integers contained in the set.

\begin{theorem} \label{thm1.4}
Suppose that $\Lambda_j,\ j=1,\ldots,d,$ are complete sets of
conjugate algebraic integers, and define
$\Lambda:=\Lambda_1\times\ldots\times \Lambda_d.$ If $\Lambda\subset
E$ for a compact set $E\subset\C^d$ with $t_{\Z}(E)<1,$ then the
integer Chebyshev polynomials for $E$ satisfy $C_n(z)=0,\
z\in\Lambda,$ for all large $n\in\N.$
\end{theorem}
A very interesting problem is how one can find and describe factors of the integer Chebyshev polynomials. Alternatively, one may ask what are the manifolds of zeros in $\C^d$ connecting points of the lattice $\Lambda$ in Theorem \ref{thm1.4}.

Following the univariate case \cite{Gor59,Mo94,BE96,FRS97,Pr1}, it is possible to give a lower bound for the multivariate integer Chebyshev constant by using the leading coefficients of minimal polynomials for algebraic numbers in the set.

\begin{theorem} \label{thm1.5}
Let $E\subset\C^d$ be a compact set. Suppose that $\Lambda_j,\
j=1,\ldots,d,$ are complete sets of conjugate algebraic numbers such
that $\Lambda:=\Lambda_1\times\ldots\times \Lambda_d \subset E.$
Denote the leading coefficient of the minimal polynomial for
$\Lambda_j$ by $a_j\in\Z.$ Assume that for each $j=1,\ldots,d$ there
are infinitely many sets $\Lambda_j(m_j)$, of cardinality
$|\Lambda_j|=m_j\to\infty$, satisfying the above stated conditions, and set
$s_j:=\limsup_{m_j\to\infty} |a_j(m_j)|^{-1/m_j}.$ Then
\[
t_{\Z}(E) \ge \prod_{j=1}^d s_j.
\]
\end{theorem}

We consider examples of the problem on polydisks, rectangles and
other special sets in the next section.

\section{Special sets}

\subsection{Polydisks}

For $r=(r_1,\ldots,r_d),$ let $D_{r}:=\{(z_1,\ldots,z_d)\in\C^d:
|z_j|\le r_j,\ r_j>0,\ j=1,\ldots,d\}$ be a polydisk in $\C^d,$
centered at the origin. This region is probably the simplest in
terms of the integer Chebyshev problem, because the extremal
polynomials are monomials, and the solution coincides with that of
the regular Chebyshev problem for ${\mathcal P}_n^d({\C})$.

\begin{proposition} \label{prop2.1}
Let $r_m:=\dis\min_{1\le j\le d} r_j,\ 1\le m\le d.$ If $r_m<1$ then
an integer Chebyshev polynomial of degree $n\in\N_0$ on $D_r$ is
$C_n(z)=z_m^n,$ with $t_{\Z}(n,D_r) = r_m^n$ and $t_{\Z}(D_r) =
r_m.$ If $r_m\ge 1$ then $C_n(z)\equiv 1,\ n\in\N_0,$ so that
$t_{\Z}(n,D_r) = t_{\Z}(D_r) = 1.$
\end{proposition}

\subsection{Rectangles}

Let $E=[a_1,b_1]\times\ldots\times[a_d,b_d]\subset\R^d$ be a (real) rectangle with faces parallel to the coordinate planes. Proposition \ref{prop1.3} and Theorem \ref{thm1.5} give the upper and the lower bounds for $t_{\Z}(E)$ in terms of one-dimensional bounds for the integer Chebyshev constant of the intervals $[a_j,b_j].$ However, it is of great interest to investigate the problem from a multivariate point of view, and determine the shape of the multivariate integer Chebyshev polynomials. We restrict our discussion to the case $d=2.$

Consider $E=[a,b]\times[c,d],$ where $a,b,c,d\in\R.$ Suppose that $\ell:[a,b]\to[c,d]$ is
a linear function with integer coefficients, i.e.,
$y=\ell(x)\in{\mathcal P}_1^1({\Z}).$ Set $F=\{(x,\ell(x)): x\in[a,b]\}.$ Let $C_n^E$ and $C_n^F$ be the integer Chebyshev polynomials of degree $n$ for $E$ and $F.$ If $C_n^E\vert_F\not\equiv 0$ then
\[
t_{\Z}(n,E) = \|C_n^E\|_E \ge \|C_n^E\|_F \ge \|C_n^F\|_F = t_{\Z}(n,F).
\]
But $C_n^F\vert_F=C_n^F(x,\ell(x))\in{\mathcal P}_n^1({\Z})$ and
\[
\|C_n^F\|_F = \|C_n^F\circ\ell\|_{[a,b]} = \|C_n^F\circ\ell\|_E \ge
t_{\Z}(n,E).
\]
Hence $t_{\Z}(n,E)=t_{\Z}(n,F)$, $C_n^E(x,y) = C_n^F(x,y)$ and
$C_n^E(x,\ell(x)) = Q_n(x)$, where $Q_n$ is a univariate integer
Chebyshev polynomial for $[a,b].$

For example, consider the square $E=[0,1]\times[0,1]$ and $y=\ell(x)=1-x.$ Numerical computations suggested the polynomial
$C_5(x,y)=xy(y-1)(x-1)(x-y)$ \cite{Me}. It does not vanish on
$F=\{(x,\ell(x)): x\in[0,1]\}$, and $C_5(x,1-x)=x^2(1-x)^2(2x-1)=Q_5(x)$, where $Q_5$ is an
integer Chebyshev polynomial for $[0,1]$ (cf. \cite{HS97}). Since
$t_{\Z}(E) > t_{\Z}(F),$ we conclude that $C_n\vert_F\equiv 0$ for
large $n.$ In fact, the numerical computation of $C_6$ through $C_9$ show that they have the factor $1-x-y$ \cite{Me}. As a consequence of Proposition \ref{prop1.3} and Theorem \ref{thm1.5}, together with \cite{Pr1} and \cite{Ch83,Mo94}, we state the bounds
\[
(0.4207)^2 < t_{\Z}([0,1]\times[0,1]) \le t_{\Z}([0,1]) < 0.4232.
\]
We plan a more detailed study of the integer Chebyshev problem for the square $[0,1]\times[0,1]$ in a forthcoming paper.

\subsection{Polylemniscates}

We consider polynomial mappings $q=(q_1,\ldots,q_d):\C^d\to\C^d,\
d\ge 2,$ such that $q_j\in{\mathcal P}_l^d({\Z})$ with $\deg(q_j)=l,\ j=1,\ldots,d.$ Furthermore, we
assume that the homogeneous parts $\hat q_j$ of degree $l$ in $q_j$
satisfy
\[
\bigcap_{j=1}^d \hat q_j^{-1}(0) = 0.
\]
The latter condition is equivalent to
\[
\liminf_{|z|\to\infty} \frac{|q(z)|}{|z|^l} > 0,
\]
where $|\cdot|$ is the Euclidean norm on $\C^d$, see Theorem 5.3.1
of \cite{Kl91}. A polynomial mapping $q$ of degree $l$ is called
simple if $\hat q_j(z)=z_j^l,\ j=1,\ldots,d.$

\begin{proposition} \label{prop2.3}
Let $q=(q_1,\ldots,q_d):\C^d\to\C^d,\ d\ge 2,$ be a simple polynomial mapping such that $q_j\in{\mathcal P}_l^d({\Z}),\ j=1,\ldots,d.$ For a polydisk $D_{r}$, $r=(r_1,\ldots,r_d),$ define the (filled-in) polylemniscate $L_r(q):=q^{-1}(D_r)$. If
$r_m=\dis\min_{1\le j\le d} r_j<1$ then an integer Chebyshev
polynomial of degree $l n$ on $L_r(q)$ is $C_{ln}(z)=q_m^n,$ with
$t_{\Z}(ln,L_r(q)) = r_m^n$ and $t_{\Z}(L_r(q)) = r_m^{1/l}.$
\end{proposition}
Note that the solutions of the integer and the regular Chebyshev problems coincide in Proposition \ref{prop2.3}. It would be of interest to find examples of sets in $\C^d,\ d>1,$ where the integer Chebyshev constants and polynomials are known explicitly, and they differ from the regular Chebyshev case. When $d=1$, such examples of lemniscates for univariate polynomials were found in \cite[Theorem 1.5]{Pr1}.

\section{A generalization of the Hilbert-Fekete upper bound}

Recall that the space of polynomials ${\mathcal P}_n^d({\C})$ of
degree at most $n$ in $d$ complex variables has the dimension
$h_n:=\binom{d+n}{n}$, which corresponds to the count of monomials
$z^k=z_1^{k_1}\ldots z_d^{k_d}$ or multi-indices
$k=(k_1,\ldots,k_d)\in \N_0^d$ with $|k|=\sum_{j=1}^d k_j \le n.$ We
arrange all multi-indices in the increasing sequence $\{k^{(i)}\},$ by following the standard lexicographic order. This order means that
$k^{(i)} \prec k^{(i+1)}$ for the multi-indices $k^{(i)}$ and $k^{(i+1)}$ if either $|k^{(i)}|\le |k^{(i+1)}|$ or $|k^{(i)}|=|k^{(i+1)}|$ and the first non-zero entry of $k^{(i)}-k^{(i+1)}$ is negative.

Given a set of points $z_i\in\C^d,\ i=1,\ldots,h_n$, we define the
Vandermonde determinant by
\[
V(z_1,\ldots,z_{h_n}):=\det(z_i^{k^{(j)}})_{i,j=1}^{h_n}.
\]
When $d=1$ and $z_i\in\C,\ i=1,\ldots,n+1$, it is well known that
\[
V(z_1,\ldots,z_{n+1})=\prod_{1\le i<j\le n+1} (z_i-z_j).
\]
However, no simple factorization formula is available for $d\ge 2$.

For a compact set $E\subset\C^d$ and $n\in\N$, define an $n$th set
of Fekete points $\{\zeta_i\}_{i=1}^{h_n}\subset E$ as maximizers
for the Vandermonde determinant:
\[
|V(\zeta_1,\ldots,\zeta_{h_n})| = \max_{\{z_i\}_{i=1}^{h_n}\subset
E} |V(z_1,\ldots,z_{h_n})|.
\]
All Fekete points change with $n$ in general, but we avoid double
indices to simplify the notation. Note that the degree of
$V(z_1,\ldots,z_{h_n})$ as a multivariate polynomial is equal to
$l_n:=d\binom{d+n}{d+1}$. Let
\[
t_{\C}(E):= \lim_{n\to\infty}
|V(\zeta_1,\ldots,\zeta_{h_n})|^{1/l_n}
\]
be the multivariate transfinite diameter of $E$ in $\C^d$. In the
univariate case, the sequence under the limit is increasing, which
immediately implies existence of the limit. Furthermore, the
transfinite diameter of $E\subset\C$ is equal to the Chebyshev
constant of $E$ and to the logarithmic capacity of $E$, as we
already mentioned is Section 1 (cf. \cite{Ra95}). The multivariate
case is much more delicate, and the existence of the defining limit
for $t_{\C}(E)$ in the general setting was established much later,
see \cite{Za75} and \cite{BC99}.

We state the following generalization of the Hilbert-Fekete estimate \eqref{1.10}.

\begin{theorem} \label{thm3.1}
For any compact set $E\subset\C^d$ that is invariant under complex
conjugation in each coordinate variable, we have
\begin{equation} \label{3.1}
t_{\Z}(E) \leq (t_{\C}(E))^{d/(d+1)}.
\end{equation}
\end{theorem}
Clearly, if $d=1$ then \eqref{3.1} yields  $t_{\Z}(E) \leq \sqrt{t_{\C}(E)}$, which is the Hilbert-Fekete inequality \eqref{1.10}. We prove Theorem \ref{thm3.1} for $E\subset\R^d$ here, to avoid a substantially more technical excursion into pluripotential theory.

\section{Proofs}

\begin{proof}[Proof of Proposition \ref{prop1.2}]

(i) Suppose that $C_n^E$ and $C_n^F$ are arbitrary integer Chebyshev
polynomials of degree $n$ for $E$ and $F.$ It follows from
Definition \ref{def1.1} that
\[
\|C_n^E\|_E \le \|C_n^F\|_E \le \|C_n^F\|_F.
\]
Hence part (i) is an immediate consequence of the same definition.

(ii) Consider $\eps>0$ and choose $n$ such that $\norm{C_n}_E^{1/n}
\le t_{\Z}(E)+\eps/2.$ Since $E$ is compact, there is a closed ball $B_R\subset\C^d$ of sufficiently large radius $R>0$ that contains $E$ strictly inside. Hence $E\subset H:=\{z\in\C^d:|C_n(z)|^{1/n} \le t_{\Z}(E)+\eps/2\} \bigcap B_R.$ On the other hand, $H\subset W:=\{z\in\C^d: |C_n(z)|^{1/n} \le t_{\Z}(E)+\eps\}\bigcap B_{2R}.$ Furthermore, the boundary of $W$ is disjoint from $H$ by the maximum modulus principle applied to $C_n$, so that we can set $\delta:=\textup{dist}(H,\partial W) =\dis \min_{z\in H, w\in \partial W} |z-w| >0.$ Hence $E_{\delta}\subset W$ and
$$ t_{\Z}(E_{\delta}) \le t_{\Z}(W) \le t_{\Z}(E)+\eps$$
by (i), where the last inequality follows by considering a sequence of polynomials $(C_n)^m,\ m\in\N,$ on $W.$ The lower bound in (ii) is an immediate consequence of (i).

\end{proof}

\begin{proof}[Proof of Proposition \ref{prop1.3}]

If we consider a univariate integer Chebyshev polynomial $Q_n$ for a
set $E_j$, then
\[
t_{\Z}(n,E)=\|C_n\|_E \le \|Q_n\|_E=\|Q_n\|_{E_j},
\]
where $C_n$ is a multivariate integer Chebyshev polynomial for $E$.
After extracting the $n$th root and passing to the limit, we obtain
that $t_{\Z}(E)\le t_{\Z}(E_j),\ j=1,\ldots,n.$

Suppose now that $t_{\C}(E_j)\ge 1,\ j=1,\ldots,d.$ Since the
Euclidean diameter diam$(E_j)\ge 2\, t_{\C}(E_j) \ge 2$ (cf.
\cite[Theorem 5.3.4]{Ra95}), we can find a point $\zeta_j\in E_j$ such that $|\zeta_j|\ge 1,\ j=1,\ldots,d.$ Substituting these values
$z_j=\zeta_j,\ j=2,\ldots,d,$ into a multivariate integer Chebyshev
polynomial $C_n$, we obtain a polynomial in one variable $z_1$ with
a leading coefficient of the form $ a \prod_{j=2}^d \zeta_j^{n_j},$
where $a$ is a nonzero integer. Note that this coefficient is at
least one in absolute value. Recall that for any monic univariate
polynomial $P_k(z_1)$ of degree $k$, one has $\|P_k\|_{E_1} \ge
(t_{\C}(E_1))^k \ge 1,$ see \cite[Theorem 5.5.4]{Ra95}. It follows that $\|C_n\|_E \ge 1.$ Since $C_n$ is an arbitrary multivariate integer Chebyshev polynomial, we obtain that $t_{\Z}(n,E) \ge 1$ and $t_{\Z}(E)\ge 1.$ Hence we can take $C_n(z) \equiv 1,$ so that $t_{\Z}(n,E) = 1$ and $t_{\Z}(E) = 1.$

\end{proof}

\begin{proof}[Proof of Theorem \ref{thm1.4}]

Suppose that $\Lambda_j=\{\lambda_{j,k}\}_{k=1}^{m_j},\
j=1,\ldots,d.$ Since non-real algebraic integers in $\Lambda_j$ come
in complex conjugate pairs, each $\Lambda_j$ is invariant under
complex conjugation, which works as a permutation of $\lambda_{j,k}.$ For any integer Chebyshev polynomial $C_n$ on $E$, consider
\[
P(z_1):=\sum_{(\lambda_2,\ldots,\lambda_d)\in\Lambda_2\times\ldots\times\Lambda_d}
|C_n(z_1,\lambda_2,\lambda_3,\ldots,\lambda_d)|^{2n},
\]
which is a polynomial in $z_1$ and $\bar z_1$ because $|C_n|^2 = C_n
\overline C_n.$ Since $C_n$ has integer coefficients, the
coefficient of $P(z_1)$ near each term $z_1^l\bar z_1^m$ is a
symmetric function of $\lambda_{j,k}\in\Lambda_j$ with integer
coefficients, for each $j=2,\ldots,d.$ Therefore, these coefficients of $P(z_1)$ are integers. From $t_{\Z}(E)<1,$ we have that $\|C_n\|_E<1$ for all sufficiently large $n\in\N.$ It also follows that $\|P\|_{E_1}<1$ for all sufficiently large $n\in\N,$ where $E_1:=\{z_1:z=(z_1,\ldots,z_d)\in E\}$ is the projection of $E$ onto the coordinate plane of $z_1.$  Hence we have for the sum
\[
\sum_{k=1}^{m_1} |P(\lambda_{1,k})|^{2n} < 1,
\]
for all sufficiently large $n\in\N.$ But the latter sum is an integer as a symmetric function in $\lambda_{1,k}$, and therefore must vanish. This means each $\lambda_{1,k}$ is a root of $P$ for $k=1,\ldots,m_1,$ so that all terms in the sum defining $P$ must vanish on the lattice $\Lambda.$

Note that if the cardinality of $\Lambda_k$ can be arbitrarily large for a certain $k$, then $C_n$ must vanish for all values of the variable $z_k$ when other variables $z_j$ are assigned values from $\Lambda_j.$ Indeed, in this case the univariate polynomial $C_n(\lambda_1,\ldots,z_k,\ldots,\lambda_d)$ in $z_k$ vanishes on $\Lambda_k$, where sets $\Lambda_k$ have an accumulation point as $|\Lambda_k|\to\infty.$

\end{proof}

\begin{proof}[Proof of Theorem \ref{thm1.5}]

For simplicity, we first assume that $E\subset\C^2$ and $C_n(z_1,z_2)$ is an integer Chebyshev polynomial for $E$. Consider the univariate polynomial
\[
P(z_1):=a_2^n \prod_{\lambda_2\in\Lambda_2} C_n(z_1,\lambda_2).
\]
The coefficients of $P(z_1)$ are symmetric functions of algebraic
numbers $\lambda_2\in\Lambda_2$, and are integers because of the
factor $a_2^n.$ Furthermore,  the number
\[
N:=a_1^{m_2 n} \prod_{\lambda_1\in\Lambda_1} P(\lambda_1) = a_1^{m_2
n} a_2^{m_1 n} \prod_{\lambda_1\in\Lambda_1,\lambda_2\in\Lambda_2}
C_n(\lambda_1,\lambda_2)
\]
is an integer by applying the fundamental theorem on symmetric forms
in a similar way. This integer cannot be zero for $m_1>m_2 n,$ since
$P$ has fixed degree at most $m_2 n$. Thus $|N|\ge 1$, and
\[
|a_1|^{m_2 n} |a_2|^{m_1 n} \|C_n\|_E^{m_1m_2} \ge |N| \ge 1.
\]
Thus the result follows after taking the power $1/(m_1 m_2 n)$ and
passing to $\limsup$'s.

In the general case $d\ge 2$, one observes that
\[
N:=\prod_{j=1}^d a_j^{n \prod_{k\neq j} m_k} \prod_{(\lambda_1,\ldots,\lambda_d)\in\Lambda_1\times\ldots\times\Lambda_d}
C_n(\lambda_1,\ldots,\lambda_d)
\]
is a nonzero integer, so that
\[
\prod_{j=1}^d |a_j|^{n \prod_{k\neq j} m_k} \|C_n\|_E^{\prod_{k=1}^d m_k} \ge |N| \ge 1.
\]

\end{proof}

\begin{proof}[Proof of Proposition \ref{prop2.1}]

Let $z^k,\ |k|\le n,$ be the leading monomial of an integer
polynomial $P_n\in{\mathcal P}_n^d({\Z})$, with the corresponding
leading coefficient $a_k\in\Z.$ It follows from Proposition 4 of
\cite[p. 428]{BC00} that
\[
\|P_n\|_{D_r} \ge |a_k| \|z^k\|_{D_r} \ge r_m^{|k|}.
\]
If $r_m<1$ then the smallest possible value for the norm is clearly attained
by the monomial $C_n(z)=z_m^n,$ so that $t_{\Z}(n,D_r) = r_m^n$ and
$t_{\Z}(D_r) = r_m.$ If $r_m\ge 1$ then $C_n(z)\equiv 1,\ n\in\N_0,$
because any other polynomial with integer coefficients has the norm
at least equal to 1 by the above estimate. Hence $t_{\Z}(n,D_r) =
t_{\Z}(D_r) = 1.$

\end{proof}

\begin{proof}[Proof of Proposition \ref{prop2.3}]

Note that $C_n(z)=z_m^n$ is both a Chebyshev and an integer Chebyshev polynomial for the polydisk $D_r$ by Proposition \ref{prop2.1}. Applying Theorem 4 of \cite{BC00}, we conclude that $C_n\circ q = q_m^n$ is both a Chebyshev and an integer Chebyshev polynomial for the polylemniscate $L_r(q).$

\end{proof}

\begin{proof}[Proof of Theorem \ref{thm3.1}]

We give a proof for $E\subset\R^d$ here. A proof of the general case involves more substantial machinery of pluripotential theory, and will be published  separately.

Suppose first that $E$ is not pluripolar in $\C^d.$ Then the
Vandermonde determinant of the Fekete points for $E$  does not
vanish, i.e., $V(\zeta_1,\ldots,\zeta_{h_n})\neq 0$ for any $n\in\N$
(see \cite{Kl91,BBCL}). We define the fundamental Lagrange interpolation polynomials in Fekete points by
\[
l_j^{(n)}(z):=\frac{V(\zeta_1,\ldots,z,\ldots,\zeta_{h_n})}
{V(\zeta_1,\ldots,\zeta_j,\ldots,\zeta_{h_n})}, \quad
j=1,\ldots,h_n,
\]
where the variable $z$ replaces $\zeta_j$ in the numerator. It is
clear from this construction that $l_j^{(n)}(\zeta_j)=1$ and
$l_j^{(n)}(\zeta_i)=0$ for $i\neq j.$ Next we express a polynomial
$P_n\in{\mathcal P}_n^d({\Z})$ as
\[
P_n(z) = \sum_{j=1}^{h_n} P_n(\zeta_j) l_j^{(n)}(z),
\]
by the Lagrange interpolation formula. Since
\[
|V(\zeta_1,\ldots,z,\ldots,\zeta_{h_n})| \le
|V(\zeta_1,\ldots,\zeta_j,\ldots,\zeta_{h_n})|, \quad z\in E,
\]
by the defining property of Fekete points, we obtain that
\[
\|l_j^{(n)}\|_E \le 1, \quad j=1,\ldots,h_n.
\]
It follows at once that
\[
\norm{P_n}_E \le \sum_{j=1}^{h_n} |P_n(\zeta_j)| \le h_n\,\max_{1\le
j\le h_n} |P_n(\zeta_j)|.
\]
Observe that
\[
f_j:=P_n(\zeta_j)=\sum_{|k|\le n} a_k \zeta_j^k, \quad
j=1,\ldots,h_n,
\]
are linear forms in $a_k$'s, with real coefficients. Applying
Minkowski's theorem (see \cite[p. 73]{Ca97}), we conclude that there
exists a set of integers $\{a_k\}_{|k|\le n}$, not all zero, such
that
\[
|f_j| \le
|\det(\zeta_i^{k^{(j)}})_{i,j=1}^{h_n}|^{1/h_n}=|V(\zeta_1,\ldots,\zeta_{h_n})|^{1/h_n}.
\]
Thus we can find a sequence of polynomials $P_n(z)=\sum_{|k|\le n}^n a_k z^k \not\equiv 0$ with integer coefficients, satisfying
\[
\norm{P_n}_E \le h_n\,|V(\zeta_1,\ldots,\zeta_{h_n})|^{1/h_n},\quad
n\in\N.
\]
Note that $\dis\lim_{n\to\infty} h_n^{1/n}=1$ and that
\[
\frac{l_n}{n\,h_n} = \frac{d}{n} \binom{d+n}{d+1}
\binom{d+n}{n}^{-1} = \frac{d}{d+1}.
\]
Hence we have that
\[
\norm{P_n}_E^{1/n} \le h_n^{1/n}
\left(|V(\zeta_1,\ldots,\zeta_{h_n})|^{1/l_n}\right)^{d/(d+1)},
\]
and \eqref{3.1} follows by passing to the limit as $n\to\infty.$

If $E$ is pluripolar in $\C^d$, then we consider the compact sets $E_m:=E\bigcup\{z=(z_1,\ldots,z_d):|z_j|\le 1/m,\ j=1,\ldots,d\}.$
Clearly, each $E_m,\ m\in\N,$ is not pluripolar \cite{Kl91}, and
$\lim_{m\to\infty} t_{\C}(E_m) = t_{\C}(E) = 0,$ see \cite[p. 287]{BC99} and \cite{Kl91}. Hence the first part of the proof and Proposition \ref{prop1.2} (i) give that
\[
t_{\Z}(E) \le t_{\Z}(E_m) \le \left(t_{\C}(E_m)\right)^{d/(d+1)} \to 0, \quad \mbox{as } m\to\infty.
\]
It follows that $t_{\Z}(E)=0$, and \eqref{3.1} is trivially satisfied. Note also that $E$ is pluripolar if and only if $t_{\C}(E)=0$, cf. \cite[p. 287]{BC99}. Thus $t_{\Z}(E)=t_{\C}(E)=0$ in this case.

\end{proof}


\end{document}